\documentclass[reqno, 12pt]{amsart}

\usepackage{amsthm,amssymb,amstext,amscd,amsfonts,amsbsy,amsxtra,latexsym}
\usepackage{fullpage}
\usepackage[latin1]{inputenc}

\newcommand{\field}[1]{\mathbb{#1}}

\def\H{\mathbb{H}}

\newcommand{\Z}{\field{Z}}
\newcommand{\R}{\field{R}}
\newcommand{\C}{\field{C}}

\newcommand{\leg}[2]{\genfrac{(}{)}{}{}{#1}{#2}}

\newcommand\spt {\operatorname{spt}}
\newcommand\sptone {\operatorname{spt1}}
\newcommand\mtwospt {\operatorname{M2spt}}

\numberwithin{equation}{section}
\newtheorem{theorem}{Theorem}[section]

\newtheorem{proposition}[theorem]{Proposition}
\theoremstyle{remark}
\newtheorem*{remark}{Remark}
\newtheorem*{ack}{Acknowledgements}

\renewenvironment{proof}[1][Proof]{\begin{trivlist}
\item[\hskip \labelsep {\bfseries #1:}]}{\qed\end{trivlist}}

\title{$\ell$-adic properties of smallest parts functions}
\author{Scott Ahlgren, Kathrin Bringmann, and Jeremy Lovejoy}
\address{Department of Mathematics\\
University of Illinois\\
Urbana, IL 61801}
\email{ahlgren@math.uiuc.edu}

\address{Mathematical Institute, University of Cologne, Weyertal 86-90, 50931 Cologne, Germany}
\email{kbringma@math.uni-koeln.de}

\address{CNRS, LIAFA, Universit\'e Denis Diderot - Paris 7, Case 7014, 75205 Paris Cedex 13, France}
\email{lovejoy@liafa.jussieu.fr}
\thanks{The second author was partially supported by NSF grant DMS-0757907 and by the Alfried Krupp prize.}

\begin{document}
\begin{abstract}
We prove explicit congruences modulo powers of arbitrary primes for three smallest parts functions: one for partitions, one for overpartitions, and one for partitions without repeated odd parts.  The proofs depend on $\ell$-adic properties of certain modular forms and mock modular forms of weight $3/2$ with respect to the   Hecke operators $T(\ell^{2m})$.
\end{abstract}
\maketitle
\section{Introduction}
Let $p(n)$ denote the number of partitions of $n$ and let $\spt(n)$ denote the number of smallest parts in the partitions of $n$.  For example, there are $7$ partitions of $5$,
\[
5,\, 4+1,\,3+2,\,3+1+1,\,2+2+1,\,2+1+1+1,\,1+1+1+1+1,
\]
and so $\spt(5) = 14$.  The generating function  \cite[Theorem 4]{An1} is given by
\[
\sum_{n \geq 1} \spt(n)q^n = \left(\prod_{n \geq 1} \frac{1}{1-q^n}\right)\left(\sum_{n \geq 1} \frac{nq^n}{1-q^n} + \sum_{n \in \mathbb{Z} \backslash \{0\}} \frac{(-1)^nq^{n(3n+1)/2}}{(1-q^n)^2}\right).
\]
The smallest parts function was introduced by Andrews \cite{An1}, who proved that
it satisfies the following Ramanujan-type congruences:
\begin{eqnarray}
\spt(5n+4) &\equiv& 0 \pmod{5}, \label{sptmod5} \\
\spt(7n+5) &\equiv& 0 \pmod{7}, \\
\spt(13n+6) &\equiv& 0 \pmod{13} \label{sptmod13} .
\end{eqnarray}

A number of studies of explicit congruences for $\spt(n)$ quickly followed Andrews' work.
Garvan \cite{Ga.5} produced explicit (and much more intricate)  congruences modulo other small primes.
Folsom and Ono  \cite{Fo-On1} studied $\spt(n)$ modulo $2$ and $3$, showing in particular
that the generating function of $\spt(n)$ is an eigenform modulo $3$ for the weight $3/2$ Hecke operators.
Garvan \cite{Ga2} showed that the generating function is in fact an eigenform modulo $72$.
 Ono \cite{On1} found systematic congruences for $\spt(n)$ modulo $\ell$ for any prime $\ell \geq 5$.
He proved that  when   $(\frac{-n}\ell)=1$ we have
\begin{equation} \label{Ono}
 \spt\left(\frac{\ell^2  n +  1}{24}\right) \equiv 0 \pmod{\ell},
\end{equation}
and that for all $n$ we have
\begin{equation} \label{Ono1}
 \spt\left(\frac{\ell^3  n +  1}{24}\right) \equiv \leg{3}{\ell} \spt\left(\frac{\ell  n +  1}{24}\right)  \pmod{\ell}.
\end{equation}

Finally,  Garvan \cite{Ga1} found systematic extensions of the congruences \eqref{sptmod5}-\eqref{sptmod13}, obtaining
results which are analogous to Ramanujan's partition congruences modulo powers of  $5$, $7$, and $11$.
If $\ell\geq 5$ is prime and $m\geq 1$, then denote by  $\delta_{\ell, m}$   the least positive integer with  $24\delta_{\ell, m}\equiv 1\pmod{\ell^m}$.
As an example of Garvan's results        modulo powers of $5$, $7$, and $13$, we have the following congruence for  each $m\geq 3$:
\begin{equation}\label{frank1}
\spt\left(5^m n+\delta_{5, m}\right)+5 \spt\left(5^{m-2}n+\delta_{5, m-2}\right)\equiv 0\pmod{5^{2m-3}}.
\end{equation}
Garvan also proved that for each of $\ell=5$, $7$, and $13$ we have
\begin{equation}\label{frank2}
\spt\left(\ell^{m} n+\delta_{\ell,m}\right)\equiv 0\pmod{\ell^{\lfloor\frac{m+1}2\rfloor}}.
\end{equation}

The works of Garvan and Ono provide elegant families of congruences which are infinite in different aspects:
the congruences \eqref{frank1} and  \eqref{frank2}  hold for arbitrary powers of certain primes, while the mod  $\ell$ congruences \eqref{Ono} and \eqref{Ono1} hold for any prime.
In this paper we  exhibit explicit congruences for $\spt(n)$ and   two other smallest parts functions which hold
modulo arbitrary powers of arbitrary primes.   The congruences satisfied by the three functions which we consider are identical in form.
In fact, they all arise for the same reason:  a modular form or mock modular form related to the generating function satisfies a simple congruence
described in terms of the Hecke operators $T(\ell^{2m})$ (see Theorem~\ref{thm1} and Proposition \ref{thm2}).  This is a general phenomenon which will hold
for a much wider class of modular forms and mock modular forms, and is of independent interest.

We begin with the results for $\spt(n)$, which extend \eqref{Ono} and  \eqref{Ono1} to any prime power.

\begin{theorem} \label{thm1bis}
For any prime $\ell \geq 5$ and any integer $m \geq 1$ the following are true:
\begin{itemize}
\item[$(i)$] If $\left(\frac{-n}{\ell}\right) = 1$, then we have
\begin{equation*}
\spt\left(\frac{\ell^{2m}n+1}{24}\right) \equiv 0 \pmod{\ell^m}.
\end{equation*}

\item[$(ii)$] For any $n$ we have
\begin{equation*}
\spt\left(\frac{\ell^{2m+1}n+1}{24}\right) \equiv \left(\frac{3}{\ell}\right) \spt\left(\frac{\ell^{2m-1}n+1}{24}\right) \pmod{\ell^m}.
\end{equation*}
\end{itemize}
\end{theorem}
\begin{remark}
Setting $n=24N-1$ in the first part of the theorem, we obtain
\[\spt\left(\ell^{2m} N+\delta_{\ell, 2m}\right)\equiv 0\pmod{\ell^m}\]
when $\leg{1-24N}\ell=1$.  So our result is implied by \eqref{frank2} for the primes $\ell=5$, $7$, and $13$.
It is not clear how (or if)  our results are related to the
more delicate \eqref{frank1} and  its counterparts modulo $7$ and $13$ (see (1.6)-(1.8) of \cite{Ga1}).
\end{remark}

\begin{remark}
Theorem  \ref{thm1bis} can be used to produce   explicit families of  congruences.  For example, we may take $n\equiv \delta\pmod \ell$ in part $(i)$ for any
$\delta$ with $\leg{-\delta}\ell=1$ to
obtain examples of the form
\[
\begin{aligned}
\spt\left(11^{2m+1}n + \frac{11^{2m}\cdot 167 + 1}{24}\right) \equiv 0 \pmod{11^{m}}, \\
\spt\left(17^{2m+1}n + \frac{17^{2m}\cdot 239 + 1}{24}\right) \equiv 0 \pmod{17^{m}}, \\
\spt\left(19^{2m+1}n + \frac{19^{2m}\cdot 287 + 1}{24}\right) \equiv 0 \pmod{19^{m}}.\\
\end{aligned}
\]
\end{remark}

Before giving the results for the other smallest parts functions, we briefly describe the ideas in the proof of Theorem~\ref{thm1bis} (the relevant definitions can be found in Section 2).
Setting $q:=e^{2\pi i \tau}$, we define the functions  $S$ and $M$ by
\begin{equation} \label{Sofqdef}
S(\tau) =\sum_{n \geq 23} \widetilde{s}(n)q^n := \sum_{n \geq 1} \spt(n)q^{24n-1}
\end{equation}
and
\begin{equation} \label{Mofqdef}
M(\tau) = \sum_{n \geq -1} \widetilde{m}(n)q^n := S(\tau) + \frac{1}{12}\sum_{n \geq 0}(24n-1)p(n)q^{24n-1}.
\end{equation}
To establish \eqref{Ono}, Ono used the fact (proved in \cite{Br}) that $M$ is a weight $3/2$ mock modular form whose shadow
is an eigenform of the weight $1/2$  Hecke operators together with the fact (proved in \cite{Br-On1}) that the Hecke operators behave nicely on such forms.
Ono demonstrated that for each prime $\ell\geq 5$ we have
\begin{equation} \label{Kencongruence}
M \big | T\left(\ell^2\right) - \left(\frac{3}{\ell}\right) M \equiv 0 \pmod{\ell},
\end{equation}
 where $T(\ell^2)$ denotes the Hecke operator of weight $3/2$ and character
$\leg{3}{\bullet}$.
In an analogous way, Theorem~\ref{thm1bis} will follow from the next result, which shows that
  \eqref{Kencongruence} is the simplest case of an $\ell$-adic property of $M$ corresponding to the  Hecke operators $T(\ell^{2m})$.
\begin{theorem} \label{thm1}
For any prime $\ell \geq 5$ and any integer $m \geq 1$ we have
\begin{equation*}
M \big | T(\ell^{2m}) - \left(\frac{3}{\ell}\right) M \big | T(\ell^{2m-2}) \equiv  0 \pmod{\ell^m}.
\end{equation*}
\end{theorem}
This theorem gives congruences for the coefficients of the mock modular form $M$, and
$\spt(n)$ inherits the congruences in Theorem  \ref{thm1bis} because of the simple form taken by $M - S$.

We turn to the  other smallest parts functions.
Recall that an \textit{overpartition} is a partition in which the first occurrence of each distinct number may be overlined.
The function  $\overline{\sptone}(n)$ (see \cite{Br-Lo-Os1}) denotes the number of smallest parts in the overpartitions of $n$ having odd smallest part.
 For example, the $14$ overpartitions of $4$ are
\begin{equation*} \label{pbarof4}
\begin{gathered}
4, \, \overline{4}, \, 3+1, \, \overline{3} + 1, \, 3 + \overline{1}, \,
\overline{3} + \overline{1}, \, 2+2, \, \overline{2}
+ 2, \, 2+1+1, \, \overline{2} + 1 + 1, \, 2+ \overline{1} + 1, \, \\
\overline{2} + \overline{1} + 1, \, 1+1+1+1, \, \overline{1} + 1 + 1 +1,
\end{gathered}
\end{equation*}
and so $\overline{\sptone}(4) = 20$.  By  \cite[Section 7]{BLO} we have
\begin{equation} \label{sptonegf}\begin{aligned}
\sum_{n \geq 1}& \overline{\sptone}(n)q^n \\
&= \left(\prod_{n \geq 1}\frac{1+q^n}{1-q^n}\right) \left(\sum_{n \geq 1}\frac{2nq^n}{1-q^{2n}} + \sum_{n \in \mathbb{Z} \backslash \{0\}} \frac{4(-1)^nq^{n^2+n}(1+q^{2n}+q^{3n})}{(1-q^{2n})(1-q^{4n})} \right).
\end{aligned}\end{equation}
Define the functions $\overline{S}$ and $\overline{M}$ by
\begin{equation*}
\overline{S} (\tau):= \sum_{n \geq 1} \overline{\sptone}(n)q^n
\end{equation*}
and
\begin{equation} \label{barM}
\overline{M}(\tau) = \sum_{n \geq 0} \overline{m}(n)q^n := \overline{S}(\tau) + \frac{1}{12}\frac{\eta(2\tau)}{\eta^2(\tau)}
\left(E_2(\tau)-4E_2(2 \tau) \right)
,
\end{equation}
where $\eta(\tau):=q^{1/24}\prod_{n \geq 1} (1-q^n)$ is Dedekind's $\eta$-function,
and $E_2$ is the weight $2$  quasimodular Eisenstein series defined by
$$
E_2(\tau) := 1-24\sum_{n \geq 1} \frac{nq^n}{1-q^n}.
$$

A priori this situation seems rather different than that encountered above.
In contrast  to $M$, the
function $\overline{M}$ is a generating function for class numbers and is thus  a Hecke eigenform (see Proposition \ref{lemma1}).
 However, there seems to be no immediate reason that $\overline{\sptone}(n)$ should
inherit congruence properties from $\overline{m}(n)$.  It turns out that  a certain weakly holomorphic modular form $\overline{g}$  related to  $\overline{M} - \overline{S}$ (see \eqref{defineg} for the definition) has the desired $\ell$-adic properties, which are described in  Proposition \ref{thm2}.
As a consequence we obtain the following  congruence properties for $\overline{\sptone}(n)$.

\begin{theorem} \label{thm2bis}
For any prime $\ell \geq 3$ and any integer $m \geq 1$ the following are true:
\begin{itemize}
\item[$(i)$] If $\left(\frac{-n}{\ell}\right) = 1$, then we have
\begin{equation}
\overline{\sptone}(\ell^{2m}n) \equiv 0 \pmod{\ell^m}.
\end{equation}

\item[$(ii)$] For any $n$ we have
\begin{equation}
\overline{\sptone}(\ell^{2m+1}n) \equiv \overline{\sptone}(\ell^{2m-1}n) \pmod{\ell^m}.
\end{equation}
\end{itemize}
\end{theorem}

\begin{remark}
The restriction to overpartitions whose smallest part is odd is essential.  While it was shown in  \cite{Br-Lo-Os1} that  both $\overline{\sptone}(n)$ and the unrestricted smallest parts function for overpartitions satisfy simple congruences modulo $3$ and $5$,  only $\overline{\sptone}(n)$ satisfies the general congruences in Theorem \ref{thm2bis}.
\end{remark}

Finally, we  consider the restriction of $\spt(n)$ to those partitions without repeated odd parts and whose smallest part is even \cite{BLO}.  Denote this function by $\mtwospt(n)$.  For example, there are $7$ partitions of $7$ without repeated odd parts,
$$
7, \,6+1, \,5+2, \,4+3, \,4+2+1, \,3+2+2, \,2+2+2+1,
$$
giving $\mtwospt(7) = 3$.  By \cite[Section 7]{BLO} we have
\begin{equation} \label{mtwosptgf}
\sum_{n \geq 1} \mtwospt(n)q^n = \left(\prod_{n \geq 1}\frac{1+q^{2n-1}}{1-q^{2n}}\right)\left(\sum_{n \geq 1} \frac{nq^{2n}}{1-q^{2n}} + \sum_{n \in \mathbb{Z} \backslash \{0\}} \frac{(-1)^nq^{2n^2+n}}{(1-q^{2n})^2}\right).
\end{equation}

Define $S2$ and $M2$ by
\begin{equation} \label{S2def}
S2(\tau) = \sum_{n \geq 7} \widetilde{s2}(n)q^n := \sum_{n \geq 0} (-1)^n\mtwospt(n)q^{8n-1}
\end{equation}
and
\begin{equation} \label{M2def}
M2(\tau) = \sum_{n \geq -1} \widetilde{m2}(n)q^{n} := S2(q) + \frac{1}{24}\frac{\eta(8\tau)}{\eta^2(16\tau)}\left(E_2(16\tau) - E_2(8\tau)\right).
\end{equation}
The situation here is very similar to  that of
 $\overline{M}$ and $\overline{S}$ above.
\begin{theorem} \label{thm3}
 For any prime $\ell \geq 3$ and any integer $m \geq 1$ the following are true:
\begin{itemize}
\item[$(i)$] If $\left(\frac{-n}{\ell}\right) = 1$, then we have
\begin{equation}
\mtwospt\left(\frac{\ell^{2m}n+1}{8}\right) \equiv 0 \pmod{\ell^m}.
\end{equation}

\item[$(ii)$] For any $n$ we have
\begin{equation}
\mtwospt\left(\frac{\ell^{2m+1}n+1}{8}\right) \equiv  \left(\frac{2}{\ell}\right)\mtwospt\left(\frac{\ell^{2m-1}n+1}{8}\right) \pmod{\ell^m}.
\end{equation}
\end{itemize}
\end{theorem}

The paper is organized as follows.  In the next section we collect some necessary background on modular forms, mock modular forms, and Hecke operators.  In Section $3$ we prove Theorem \ref{thm1} and use it to obtain the congruences in Theorem~\ref{thm1bis}.  In Section $4$ we prove       Theorem~\ref{thm2bis} and in Section $5$ we prove Theorem \ref{thm3}.

\begin{ack} The authors thank Frank Garvan and Ken Ono for helpful comments on an earlier version of this paper.
\end{ack}
\section{Preliminaries}

First we recall the definitions of harmonic weak Maass forms and mock modular forms (see, for example, \cite{On2} or \cite{Za1} for details).
If $k\in \frac{1}{2}\Z\setminus \Z$ and  $\tau=x+iy$ with $x, y\in \R$,
then the  weight $k$ \textit{hyperbolic Laplacian} is given by
\begin{equation*}\label{laplacian}
\Delta_k := -y^2\left( \frac{\partial^2}{\partial x^2} +
\frac{\partial^2}{\partial y^2}\right) + iky\left(
\frac{\partial}{\partial x}+i \frac{\partial}{\partial y}\right).
\end{equation*}
If $d$ is odd, then define $\epsilon_d$ by
\begin{equation*}
\epsilon_d:=\begin{cases} 1 \ \ \ \ &{\text {\rm if}}\ d\equiv
1\pmod 4,\\
i \ \ \ \ &{\text {\rm if}}\ d\equiv 3\pmod 4. \end{cases}
\end{equation*}
If    $4\mid N$ and $\chi$ is  a Dirichlet character  modulo $N$, then
 a {\it harmonic weak Maass form of weight $k$ with Nebentypus $\chi$ on
$ \Gamma_0(N)$} is a  smooth function $F:\H\to \C$
satisfying the following:
\begin{enumerate}
\item For all $ \left(\begin{smallmatrix}a&b\\c&d
\end{smallmatrix} \right)\in \Gamma_0(N)$ and all $\tau \in \H$, we
have
\begin{displaymath}
F \left(\frac{a \tau +b}{c \tau +d} \right)
= \leg{c}{d}\epsilon_d^{-2k} \chi(d)\,(c\tau+d)^{k}\ F(\tau).
\end{displaymath}
\item  $\Delta_k(F)=0$.
\item The function $F$ has
at most linear exponential growth at the cusps of $\Gamma_0(N)$.
\end{enumerate}
We denote by $H_{k}\left(\Gamma_0(N), \chi\right)$ the space of harmonic weak Maass forms of weight $k$ with Nebentypus $\chi$ on  $\Gamma_0(N)$.
We denote the  subspace of weakly holomorphic forms (i.e., those meromorphic forms whose poles are supported at the cusps of $\Gamma_0(N)$) by  $M_{k}^!\left(\Gamma_0(N), \chi\right)$, and the space of holomorphic forms by $M_{k}\left(\Gamma_0(N), \chi\right)$   (if $\chi$ is trivial, then we drop it from the notation).
Each harmonic weak Maass form $F$ decomposes uniquely  as the sum of  a holomorphic part $F^{+}$ and a non-holomorphic part $F^{-}$.
The holomorphic part, which is  known as a \textit{mock modular form},  is a power series in $q$ with at most finitely many negative exponents.

We will  suppose throughout that $\chi$ is a quadratic character.  For $m$ coprime  to $N$, the   weight $3/2$ Hecke operators $T_{\chi}(m^2)$
act on the spaces   $M_{k}^!\left(\Gamma_0(N), \chi\right)$,   $M_{k}\left(\Gamma_0(N), \chi\right)$,
and $H_{k}\left(\Gamma_0(N), \chi\right)$ (see Section 7 of \cite{Br-On1} for a discussion of the latter).
For primes $\ell\nmid N$ the action on $q$-expansions
 is given by
\begin{equation} \label{Heckeaction}
\left. \left( \sum_n a(n) q^n\right)\right| T_{\chi}(\ell^{2})
:= \sum_n \left(a \left( \ell^2n\right) + \chi(\ell)\leg{-n}{\ell}a(n)+ \ell a \left( \frac{n}{\ell^2}  \right) \right)q^n.
\end{equation}
We will drop the subscript $\chi$ from $T_{\chi}$ when there can be no confusion.
The Hecke operators commute, and
the multiplicative relations are the same as those in weight $2$ and trivial character.  In particular,
for $m\geq 2$  we have
\begin{equation} \label{higherHeckedef}
\begin{split}
T(\ell^{2m})&=  T(\ell^{2m-2})T(\ell^2) - \ell \cdot T(\ell^{2m-4})\\
&=  T(\ell^{2})T(\ell^{2m-2}) - \ell \cdot T(\ell^{2m-4}).
\end{split}
\end{equation}

The next proposition describes the action of  $T(\ell^{2m})$  on Fourier expansions.
We begin with the  series
\begin{equation*}\label{F0define}
F_0 (\tau)=\sum_n a_0(n)q^n.
\end{equation*}
If  $m\geq 1$  and $\chi$ is a quadratic character, define the series $F_{m,\chi}$ by
\begin{equation}
F_{m,\chi}(\tau)=\sum_n a_m(n)q^n:=F_{0}\big| T_{\chi}(\ell^{2m})(\tau)- \chi(\ell) F_{0}\big| T_{\chi}(\ell^{2m-2})(\tau). \label{Fdef}
\end{equation}
From  \eqref{higherHeckedef} it follows that for each  $m\geq 2$ we have
\begin{equation}
F_{m,\chi}=F_{m-1,\chi}\big | T_{\chi}(\ell^2)-\ell \cdot F_{m-2,\chi}.\label{Frec}
\end{equation}

\begin{proposition}  \label{higherHeckeaction}
With notation as above,  for each $m\geq 1$  let $F_{m,\chi}(\tau)=\sum_n a_m(n)q^n$ be defined  as in \eqref{Fdef}.
\begin{itemize}
\item[$(i)$]
For any $n$, and  for all $m\geq1$, we have
\[
a_m\left(\ell^2 n\right)-\ell a_{m-1}(n)=a_0\left(\ell^{2m+2}n\right)-\chi(\ell) a_0\left(\ell^{2m} n\right) .
\]
\item[$(ii)$]
If $\ell\nmid n$, then  for all $m \geq 1$ we have
\[
a_m(n)=a_0\left(\ell^{2m}n\right)+
\left( 1- \left(\frac{-n}{\ell}\right) \right)\cdot \sum_{k=1}^m  (-1)^k \chi(\ell)^k a_0\left(\ell^{2m-2k}n\right).
\]
\item[$(iii)$]
If $\ell\mid\mid n$, then  for all $m\geq 1$ we have
\[
a_m(n)=a_0\left(\ell^{2m} n\right)-\chi(\ell) a_0\left(\ell^{2m-2}n\right).
\]
\end{itemize}
\end{proposition}

\begin{proof}
Using  \eqref{Heckeaction} and \eqref{Fdef} we find for any $m\geq 1$ that
\begin{equation}\label{mequals1-0}
a_{1}\left(\ell^{2m } n\right) =a_0\left(\ell^{2m+2}n\right)+\ell a_0\left(\ell^{2m-2}n\right)-\chi(\ell) a_0\left(\ell^{2m}n\right).
\end{equation}
In particular the first formula  is true when $m=1$.
Suppose that $m\geq 2$. Repeated application of \eqref{Frec} together with \eqref{Heckeaction} shows that
\[
a_m(\ell^2n)-\ell a_{m-1}(n)
=a_{1}(\ell^{2m } n) -\ell a_{0}(\ell^{2m-2} n),
\]
and the first formula follows by \eqref{mequals1-0}.

Suppose now that $\ell^2\nmid n$.  The second  and third formulas   can be checked directly when $m=1$ using
\eqref{Heckeaction} and  \eqref{Fdef}.
 When $m\geq 2$, we use  \eqref{Heckeaction}, \eqref{Frec},   and the first formula  to  obtain
\[\begin{aligned}
a_m(n)=&a_{m-1}\left(\ell^2 n\right)+\left(\frac{-n}{\ell}\right)\chi(\ell)a_{m-1}(n)-\ell a_{m-2}(n)\\
=&a_0\left(\ell^{2m} n\right)+\left(\frac{-n}{\ell}\right)\chi(\ell)a_{m-1}(n)-\chi(\ell)  a_0\left(\ell^{2m-2}n\right).
\end{aligned}
\]
This proves the third formula. The   second formula   follows after induction.
 \end{proof}

\section{ The $\spt$ function and the proofs of Theorems~\ref{thm1bis} and \ref{thm1}}
We begin  by proving Theorem \ref{thm1}.

\begin{proof}[Proof of Theorem \ref{thm1}]
Correcting a sign error in equation (1.6) of \cite{Br}, we deduce the following (which is also recorded as Thm. 2.1 of \cite{On1}):
\begin{equation} \label{MaassTheorem}
\widehat{M}(\tau): = M(\tau)
+ NH(\tau) \in H_{\frac32}\left(\Gamma_0(576), \left(\frac{12}{\bullet}\right)\right),
\end{equation}
where
\begin{equation*}
NH(\tau) :=
\frac{-i}{4 \pi \sqrt{2}}
\int_{- \overline{\tau}}^{i \infty} \frac{\eta(24w)}{\left( -i\left(\tau+w \right)^{\frac{3}{2}}\right)} dw .
\end{equation*}
If $\ell\geq 5$ is prime, then $\eta(24z)$ is an eigenform of the Hecke operator of index $\ell^2$ and weight $1/2$
whose  eigenvalue is $\leg{3}{\ell}(1+\ell^{-1})$.
By  Lemma $7.4$ of \cite{Br-On1} it follows that
\[
\widehat{M}  | T(\ell^2) = M | T(\ell^2) +  \leg{3}{\ell}(1+\ell)\cdot NH  \in H_{\frac32}\left(\Gamma_0(576), \left(\frac{12}{\bullet}\right)\right).
\]
Arguing inductively using \eqref{higherHeckedef} we obtain, more generally,
\begin{equation} \label{shadoweigen}
\begin{aligned}
\widehat{M}  \left| T(\ell^{2m}) \right. = M  \left| T(\ell^{2m})\right. + \leg{3}{\ell}^m &\left(1+\ell+\ell^2+\cdots+\ell^m\right)\cdot NH \\
&\qquad \in H_{\frac32}\left(\Gamma_0(576), \left(\frac{12}{\bullet}\right)\right).\\
\end{aligned}
\end{equation}
In order to work with  integral coefficients, we introduce  the function
\[
M^*(\tau):= -12 M(\tau)= \frac 1q-35q^{23}+ \cdots.
\]
Defining
\begin{equation}
G_{\ell,m}:= M^*|T(\ell^{2m}) - \leg{3}{\ell} M^*|T(\ell^{2m-2})
-\leg{3}{\ell}^m \ell^m M^*,
\end{equation}
we see using  \eqref{shadoweigen} that
\begin{equation} \label{Glmweakly}
G_{\ell,m} \in M^{!}_{\frac32}\left(\Gamma_0(576),\leg{12}{\bullet}\right).
\end{equation}
Another induction argument using \eqref{Heckeaction}  and  \eqref{higherHeckedef} shows that for each $m\geq 0$
we have
\[
M^*\left| T\left( \ell^{2m}\right)(\tau)\right.
= \frac{\ell^m}{q^{\ell^{2m}}}  + \leg{3}{\ell} \frac{\ell^{m-1}}{q^{\ell^{2m-2}}}
+ \leg{3}{\ell}^2 \frac{\ell^{m-2}}{q^{\ell^{2m-4}}}
+ \cdots +\leg{3}{\ell}^m \frac{1}{q} + O \left(q^{23} \right).
\]
It follows that for $m\geq 1$  we have
\begin{equation} \label{Mordervanish}
G_{\ell,m}(\tau) = \frac{\ell^m}{q^{\ell^{2m}}}
- \leg{3}{\ell}^m \cdot  \frac{\ell^m}{q}
+ O \left(q^{23} \right).
\end{equation}
Now define the functions
\[
F_{\ell,m}(\tau):= \eta^{\ell^{2m}} (24 \tau)
G_{\ell,m}(\tau)
\]
and
\begin{equation} \label{Hdef}
H_{\ell,m}(\tau):= F_{\ell,m}\left(\frac{\tau}{24} \right).
\end{equation}

We wish to show that
\begin{equation} \label{holmod}
H_{\ell,m}    \in M_{\frac{\ell^{2m}+3}{2}}(\Gamma_0(1)) .
\end{equation}
Using \eqref{Glmweakly} together with the fact that $\eta(24z) \in M_{1/2}(\Gamma_0(576),\leg{12}{\bullet})$, we see that  $F_{\ell,m}$ is an element of $M^{!}_{\frac{\ell^{2m}+3}{2}}(\Gamma_0(576))$.
The coefficients of $F_{\ell,m}$ are supported on exponents divisible by $24$, which implies that
\[
H_{\ell,m}\in M^!_{ \frac{\ell^{2m}+3}{2}} \left(\Gamma_0(24)\right).
\]
Recall that the Fricke involution $W_N$ on $M^!_{k}(\Gamma_0(N))$ with $k \in 1/2 + \Z$  (see, for example,  \cite{Ono}) is defined by
\begin{equation}\label{Fricke}
F\left|_{k} W_{N}(\tau) \right.
:= \left( -i \sqrt{N} \tau\right)^{-k} F \left( -\frac{1}{N \tau}\right).
\end{equation}
We argue as in the proof of Theorem~2.2 of \cite{On1}.  Using the fact that $M$ has eigenvalue $-1$ under $W_{576}$ together with
the fact that $\eta(-1/\tau)=\sqrt{-i \tau} \cdot \eta(\tau)$ and the fact that $W_{576}$ commutes with the Hecke operators $T(\ell^2)$ in question,
we compute that
\[
H_{\ell, m}(-1/24\tau)=(24\tau)^\frac{\ell^{2m}+3}2 H_{\ell, m}(24\tau).
\]
Replacing $\tau$ by $\tau/24$, we conclude that \eqref{holmod} is indeed true.

Now from \eqref{Mordervanish} we see that $H_{\ell,m}$ vanishes modulo $\ell^m$ to order at least $\frac{\ell^{2m} +23}{24}$.  By Sturm's criterion \cite{Sturm},
$H_{\ell, m}$ is identically zero modulo $\ell^m$.

\end{proof}

\begin{proof}[Proof of Theorem \ref{thm1bis}]
We use  the notation of Section~2   with $F_0(\tau)=M(\tau)=\sum_n\widetilde m(n)q^n$ and $\chi(\ell)=\leg{3}{\ell}$.
Observe that Theorem~\ref{thm1} (with $a_m$ defined as in    \eqref{Fdef}) gives
\[a_m(n)\equiv 0\pmod {\ell^m}\ \ \ \  \text{for all $n, m\geq 1$}.\]
By Proposition \ref{higherHeckeaction} (ii)
we see that if  $\leg{-n}{\ell} = 1$, then

\[
\widetilde{m}\left(\ell^{2m}n\right) \equiv 0 \pmod{\ell^m}.
 \]
  Noting that the coefficient of  $q^n$ in $M- S$ is divisible by $n$, we conclude that
\[
  \widetilde{s}\left(\ell^{2m}n\right) \equiv 0 \pmod{\ell^m},
\]
 which is the first assertion in Theorem~\ref{thm1bis}.

 Proposition ~\ref{higherHeckeaction} (iii) gives, for $m\geq 1$ and  $\ell\nmid n$,
\begin{equation}\label{tildemcong}
\widetilde{m}\left(\ell^{2m+1}n\right) - \left(\frac{3}{\ell}\right) \widetilde{m}\left(\ell^{2m-1}n\right) \equiv 0 \pmod{\ell^m}.
\end{equation}
Part (i) of that proposition shows that \eqref{tildemcong} holds when $\ell\mid n$ as well.
We may again replace $\widetilde{m}(n)$ by $\widetilde{s}(n)$, and  Theorem \ref{thm1bis} follows.

\end{proof}

\begin{remark}
There are congruence properties for $\widetilde{m}(n)$ which are not inherited by $\widetilde{s}(n)$.  For instance, applying Proposition \ref{higherHeckeaction} (ii) with $\leg{-n}{\ell} = -1$ to Theorem \ref{thm1}  gives
\[
\widetilde{m}\left(\ell^{2m}n\right) - 2\leg{3}{\ell}\widetilde{m}\left(\ell^{2m-2}n\right) + 2 \widetilde{m}\left(\ell^{2m-4}n\right) \cdots +2(-1)^m{\leg{3}{\ell}}^m \widetilde{m}(n) \equiv 0 \pmod{\ell^m},
\]
which is not guaranteed to hold with $\widetilde{s}(n)$ in place of $\widetilde{m}(n)$.
\end{remark}

\section{Proof of Theorem \ref{thm2bis}}
We now turn to the congruences for $\overline{\sptone}(n)$.  First we need the following proposition.
\begin{proposition} \label{lemma1}
Let $\overline{M}$  be defined as in \eqref{barM}.  Then for each odd prime $\ell$, $\overline{M}$  is an eigenform for the Hecke operator $T(\ell^2)$  of weight $3/2$ and trivial character, with eigenvalue $\ell+1$. Moreover, for all $m \geq 1$ we have
\begin{equation} \label{overlineMladic}
\overline{M} \big | T\left(\ell^{2m} \right)- \overline{M} \big | T\left(\ell^{2m-2}\right) \equiv 0 \pmod{\ell^m}.
\end{equation}
\end{proposition}
\begin{proof}
In the proof of Theorem 1.4 in \cite{Br-Lo-Os1} it is shown that
\[
 \overline{M} + \overline{NH}\in H_{\frac32}\left( \Gamma_0(16)\right),
\]
where
\[
\overline{NH}(\tau):=
 \frac{1}{2 \sqrt{2}\pi i} \int_{-\overline{\tau}}^{i \infty}
 \frac{\eta^2(w)}{\eta(2w)\left(-i(\tau+w) \right)^{\frac32}} dw.
 \]
  Define $\overline{f}$ by
\begin{equation} \label{fdef}
\overline{f}(\tau) := 4\cdot \frac{\eta(2\tau)}{\eta^2(\tau)}\sum_{n \in \mathbb{Z}} \frac{(-1)^nq^{n^2+n}}{\left(1+q^n\right)^2}.
\end{equation}
In Theorem 1.1 of \cite{Br-Lo.5} it is shown that
\[
 \overline{f} - 4 \cdot \overline{NH} \in H_{\frac32}\left( \Gamma_0(16)\right).
\]

  Canceling the non-holomorphic parts, we conclude that $4 \overline{M} +  \overline{f}\in M^!_\frac32(\Gamma_0(16))$.  Using \eqref{fdef} and \eqref{sptonegf},
we see that the coefficients of
\[
\frac{\eta^2(\tau)}{\eta(2\tau)}\left(4\overline{M}(\tau)+\overline{f}(\tau)\right)
\]
have at most polynomial growth.  A weakly holomorphic modular form whose coefficients grow at most polynomially must be holomorphic.
We then prove that  $4\overline{M}+\overline{f} = 0$  by checking sufficiently many coefficients.

Now by \cite[Prop. 5.1]{Br-Lo-Os1},
$\overline{f}$ is a generating function for class numbers and an eigenform for the weight $3/2$ Hecke operator $T(\ell^2)$ with eigenvalue $\ell+1$.  Hence the same is true for $\overline{M}$.
The property \eqref{overlineMladic}  then follows by induction using the second identity in  \eqref{higherHeckedef}.

\end{proof}
\begin{proof}[Proof of Theorem \ref{thm2bis}]
Arguing as in the proof of Theorem \ref{thm1bis}, and noting that the character is trivial, we obtain the following congruences for $m \geq 1$:
\begin{equation}\label{mbarcong}\begin{aligned}
\overline{m} \left(\ell^{2m}n\right) &\equiv  0 \pmod{\ell^m} \quad\quad \text{if \ $\left(\frac{-n}{\ell}\right) = 1$},\\
\overline{m}\left(\ell^{2m+1}n\right) - \overline{m}\left(\ell^{2m-1}n\right) &\equiv  0 \pmod{\ell^m}.\\
\end{aligned}\end{equation}
In order to replace $\overline{m}$ by $\overline{\sptone}$, we must show that the coefficients of  $\overline{M} - \overline{S}$ satisfy the same congruences.
We recall that the  generating function for overpartitions is given by
\[
\overline{P}(\tau) = \sum_{n \geq 0} \overline{p}(n)q^n := \frac{\eta(2 \tau)}{\eta^2(\tau)}.
\]
Differentiating gives
\[
q\frac{d}{dq} \, \overline{P}(\tau) = \sum_{n\geq 1} n\overline{p}(n)q^n=\frac{1}{12}
\frac{\eta(2 \tau)}{\eta^2(\tau)} \Big(E_2(2\tau)-E_2(\tau)\Big).
\]
With
\[
E(\tau):=2E_2(2\tau)-E_2(\tau)\in M_2\Big(\Gamma_0(2)\Big)
\]
and
\begin{equation} \label{hdef}
\overline{h}(\tau):=E(\tau) \frac{\eta(2\tau)}{\eta^2(\tau)}\in M^!_{\frac{3}{2}}\Big(\Gamma_0(16)\Big),
\end{equation}
 we see that
\begin{equation}\label{quarterh}
 \overline{M}(\tau)- \overline{S}(\tau)= 2q\frac{d}{dq}\, \overline{P}(\tau) -\frac14 \cdot \overline h(\tau).
 \end{equation}
It is clear that  the congruences \eqref{mbarcong}    are  satisfied with $\overline{m}(n)$ replaced by $n\overline p(n)$.
Therefore, the key step will be to prove the following
\begin{proposition} \label{thm2}
For any prime $\ell \geq 3$ and any integer $m \geq 1$, we have
\[
\overline{h} | T\left(\ell^{2m}\right) - \overline{h} | T\left(\ell^{2m-2}\right) \equiv 0 \pmod{\ell^m}.
\]
\end{proposition}
Assuming for the moment that this proposition is true, and writing $\overline h(\tau)=\sum_n c(n) q^n$,
 we conclude  as before that the congruences \eqref{mbarcong} are   satisfied with $\overline{m}(n)$ replaced by $c(n)$.
Then \eqref{quarterh} shows that the congruences are also satisfied by the coefficients of $\overline{S}$, and
Theorem \ref{thm2bis} follows.
\end{proof}

\begin{proof}[Proof of Proposition~\ref{thm2}]
We define $\overline g\in M^!_{\frac32}(\Gamma_0(16))$ by
\begin{equation} \label{defineg}
\overline{g}(\tau):=E(8\tau)\frac{\eta(8\tau)}{\eta^2(16\tau)}=q^{-1}+O\left(q^7\right).
\end{equation}
Recalling the definition \eqref{Fricke} of the involution   $W_{16}$, we find that
\begin{equation} \label{hFrickeg}
\overline{h}|_{\frac{3}{2}}W_{16}(\tau)=(-4i\tau)^{-\frac{3}{2}}\overline{h}\left(-\frac{1}{16\tau}\right)=\sqrt{8}\cdot
\overline{g}(\tau).
\end{equation}
 If $\ell\neq 2$ is prime, then      $W_{16}$   preserves $\ell$-integrality  and commutes with the operators
  $T(\ell^{2m})$.
Therefore Proposition~\ref{thm2} will follow if we can prove that
for     any prime $\ell \geq 3$ and any integer $m \geq 1$, we have
\begin{equation}\label{gclaim}
\overline{g} \big | T(\ell^{2m}) - \overline{g} \big | T(\ell^{2m-2}) \equiv 0 \pmod{\ell^m}.
\end{equation}

The rest of the section is devoted to proving \eqref{gclaim}.
 To begin, define
\begin{equation}\label{glmdef}
\overline{G}_{\ell,m}:=
\overline{g} | T(\ell^{2m})  - \overline{g} | T(\ell^{2m-2})   \in M_{\frac32}^!\left(\Gamma_0(16) \right),
\end{equation}
and
\begin{equation} \label{Flmdef}
\overline{F}_{\ell,m}(\tau):=
\overline{G}_{\ell,m}(\tau)
\frac{\eta^{2\ell^{2m}}(16\tau)}{\eta^{\ell^{2m}}(8\tau)}
\in M_{\frac{\ell^{2m}+ 3}{2}}^!\left( \Gamma_0(16)\right).
\end{equation}
 Then it will suffice to
 show that
\begin{equation}\label{Fcongred}
\overline{F}_{\ell,m}\equiv 0\pmod{\ell^{m}}.
\end{equation}

Using    \eqref{defineg} and \eqref{higherHeckedef}, we conclude by induction that
\[
\overline{g} \left| T\left( \ell^{2m}\right)(\tau)\right.
= \frac{\ell^m}{q^{\ell^{2m}}}  +  \frac{\ell^{m-1}}{q^{\ell^{2m-2}}}
+  \cdots
+  \frac{1}{q} + O \left(q^{7} \right),
\]
which in turn gives
\[
\overline{G}_{\ell,m}(\tau)
=\frac{\ell^m}{q^{\ell^{2m}}}
+O\left(q^7\right).
\]
Moreover,  $\overline{G}_{\ell,m}$
is supported on exponents which are $\equiv 7\pmod{8}$.
Since the form
\[
\frac{\eta^{2\ell^{2m}}(16\tau)}{\eta^{\ell^{2m}}(8\tau)}=q^{\ell^{2m}}+O\Big(q^{\ell^{2m}+8}\Big)
\]
is supported on exponents which are $\equiv 1\pmod{8}$, we conclude that
\begin{equation} \label{inforder}
\overline{F}_{\ell,m}(\tau) =\ell^m+ O\left(q^{\ell^{2m}+7}\right)
\end{equation}
is supported on exponents which are  $\equiv 0\pmod{8}$.
Setting
\begin{equation} \label{Hlmdef}
\overline{H}_{\ell,m}(\tau):=\overline{F}_{\ell,m}\left(\tau/8\right),
\end{equation}
we find that
$\overline{H}_{\ell,m} \in M^!_{\frac{\ell^{2m}+3}{2}} (\Gamma_0(2)).$
It is clear that $\overline{H}_{\ell,m}$ is holomorphic on $\field{H}$ and at $\infty$.
A computation involving \eqref{Hlmdef}, \eqref{Flmdef}, \eqref{glmdef}, \eqref{hFrickeg} and \eqref{hdef}, together with the fact that the Fricke involution \eqref{Fricke} commutes with the Hecke operators $T(\ell^{2m})$,
shows that
$
(\sqrt{2}\tau)^{-\left(\frac{\ell^{2m}+3}{2}\right)}\overline{H}_{\ell,m}\left(\frac{-1}{2\tau}\right)
$
is holomorphic at $\infty$.  Thus $\overline{H}_{\ell,m}$ is holomorphic at the cusp $0$, and
we conclude that $\overline{H}_{\ell,m} \in M_{\frac{\ell^{2m}+3}{2}} (\Gamma_0(2))$.
Using \eqref{inforder}, we see that $\overline{H}_{\ell,m}$ vanishes modulo $\ell^m$  at infinity to order at least $\frac{\ell^{2m}+7}{8}$.
By Sturm's criterion \cite{Sturm}, $\overline{H}_{\ell,m}$ vanishes identically modulo $\ell^{m}$.  This establishes \eqref{Fcongred}, and so finishes the proof of Proposition~\ref{thm2}.

 \end{proof}

\section{Proof of Theorem \ref{thm3}}

We begin with a proposition.
\begin{proposition} \label{lemma2}
Let $M2$ be defined as in \eqref{M2def}.  Then for each odd prime $\ell$, $M2$ is an eigenform for the Hecke operator $T(\ell^2)$ of weight $3/2$ and trivial character,
with eigenvalue $\ell+1$.  Moreover, for all $m \geq 1$ we have
\begin{equation} \label{M2ladic}
M2 | T\left(\ell^{2m}\right) - M2 | T\left(\ell^{2m-2}\right) \equiv 0 \pmod{\ell^m}.
\end{equation}
\end{proposition}

\begin{proof}
From Theorems 6.3 and 7.1  and equation (1.2) of \cite{BLO}
we know that
\begin{equation}\label{M2maass}
M2
+\overline{NH_2} \in H_{\frac32}\left(\Gamma_0(16)\right),
\end{equation}
where
\[
\overline{NH_2}(\tau):=
\frac{1}{4 \sqrt{2} \pi i} \int_{- \overline{\tau}}^{i \infty}
\frac{\eta^2(16 w)}{\eta(8w)\left(-i (\tau+w) \right)^{\frac32}} dw .
\]
In fact, \eqref{M2maass} remains true with  $M2(\tau)$ replaced by the generating function
\[
\sum_{n \geq 1} H(8n-1)q^{8n-1},
\]
where $H(n)$ is the Hurwitz class number.  To see this, recall the Zagier-Eisenstein series \cite[Section 2.2]{Hi-Za1},
\begin{equation} \label{ZagierEisenstein}
\mathcal{F}(\tau): = \sum_{n \equiv 0,3 \pmod{4}} H(n)q^n + \frac{1}{8 \sqrt{2} \pi i} \int_{- \overline{\tau}}^{i \infty}
\frac{\Theta(w)}{\left(-i (\tau+w) \right)^{\frac32}} dw \in H_{\frac32}(\Gamma_0(4)),
\end{equation}
where $\Theta(\tau):= \sum_{n\in \Z}q^{n^2}$.
Then
\[
\frac{1}{2}\left(F(\tau)-F(\tau+1/2)\right) =
\sum_{n \equiv 3 \pmod{4}} H(n)q^n +
\overline{NH_2}(\tau)  \in H_{\frac32}(\Gamma_0(16)).
\]
It is known that the generating function for $H(8n+3)$ is a modular form (see, for example,  pages $92-93$ of \cite{Hi-Za1}).  In fact, we have
\[
\sum_{n \equiv 3 \pmod{8}} 3H(n)q^{n} = \frac{\eta^6(16\tau)}{\eta^3(8\tau)} \in M_{\frac32}(\Gamma_0(16)),
\]
and so
\begin{equation} \label{Zagier2}
\sum_{n \equiv 7 \pmod{8}} H(n)q^n +\overline{NH_2}(\tau) \in H_{\frac32}(\Gamma_0(16)),
\end{equation}
which is the desired counterpart to \eqref{M2maass}.  Canceling the non-holomorphic parts in \eqref{M2maass} and \eqref{Zagier2}, and arguing as in the proof of Proposition \ref{lemma1}, we find that
\begin{equation} \label{M2equalsH}
M2(\tau) = \sum_{n \geq 1} H(8n-1)q^{8n-1}.
\end{equation}

Next let $\mathcal{H}(\tau)$ denote the generating function for $H(n)$ in \eqref{ZagierEisenstein}.  Recall that for each odd prime $\ell$, $\Theta(\tau)$ is an eigenform of the Hecke operator of index $\ell^2$ and weight $1/2$ whose eigenvalue is $(1+\ell^{-1})$.
An application of Lemma $7.4$ of \cite{Br-On1} gives that
\[
\mathcal{H} | T(\ell^2) -  (1+\ell) \mathcal{H} \in  M_{\frac{3}{2}}^{!}(\Gamma_0(4)).
\]
Since $H(n)$ grows polynomially we may replace $M_{\frac{3}{2}}^{!}(\Gamma_0(4))$ by $M_{\frac{3}{2}}(\Gamma_0(4))$ above.   The latter is a one-dimensional space generated by $\Theta^3(\tau)$.  Comparing constant terms we conclude that $\mathcal{H} | T(\ell^2)  = (1+\ell) \mathcal{H}$ and by \eqref{M2equalsH} the same is then true for $M2(\tau)$ in place of $\mathcal{H}(\tau)$.  
The property \eqref{M2ladic}  then follows from \eqref{higherHeckedef} using induction.
\end{proof}
We now turn to the proof of the congruences for $\mtwospt$ in Theorem \ref{thm3}.
\begin{proof}[Proof of Theorem \ref{thm3}]
Arguing as before, we obtain for $m \geq 1$ the congruences
\begin{equation} \label{m2cong}
\begin{aligned}
\widetilde{m2}\left(\ell^{2m}n\right) &\equiv 0 \pmod{\ell^m} \quad\quad \text{if\  $\left(\frac{-n}{\ell}\right) = 1$},\\
\widetilde{m2}\left(\ell^{2m+1}n\right) - \widetilde{m2}\left(\ell^{2m-1}n\right) &\equiv  0 \pmod{\ell^m}.\\
\end{aligned}
\end{equation}
In order to replace $\widetilde{m2}$ by $\widetilde{s2}$ (and then by $\mtwospt$), we must show that $M2 - S2$ satisfies the same congruences.
To   this end, we start with the generating function for partitions without repeated odd parts, which is given by
\[
R(\tau) = \sum_{n \geq 0}(-1)^{n}M2(n)q^{8n-1} := \frac{\eta(8\tau)}{\eta^2(16\tau)}.
\]
Differentiating, we obtain
\[
q\frac{d}{dq}R(\tau)   = \sum_{n \geq 0} (8n-1) (-1)^n M2(n)q^{8n-1} =
\frac{\eta(8 \tau)}{\eta^2(16 \tau)}
\left(\frac{-4E_2(16\tau)+E_2(8\tau)}{3}\right).
\]
We conclude that
\[
M2(\tau)-S2(\tau)  = \frac{1}{16}\left(\overline{g}(8\tau)  +  q\frac{d}{dq}R(\tau)\right),
\]
where $\overline{g}$ was defined in \eqref{defineg}.  By \eqref{gclaim} it follows that the congruences \eqref{m2cong} continue to hold  with $\widetilde{m2}(n)$ replaced by
\[\widetilde{s2}(n) = (-1)^{\frac{n+1}{8}}\operatorname{M2spt}\left(\frac{n+1}{8}\right),\]
 and so the theorem is proved.
\end{proof}

\end{document}